\documentclass[11pt,reqno]{amsart}

\usepackage{amssymb, amsmath, amsthm}
\usepackage{hyperref}
\usepackage[alphabetic,lite]{amsrefs}
\usepackage{verbatim}
\usepackage{amscd}   
\usepackage[all]{xy} 
\usepackage{youngtab} 
\usepackage{young} 
\usepackage{ytableau}
\usepackage{tikz}
\usepackage{ mathrsfs }
\usepackage{cases}
\usepackage{array}
\usepackage{tabu}
\usepackage{calligra,mathrsfs}
\usepackage{dsfont}
\usepackage{upgreek}

\newcommand{\defi}[1]{{\upshape\sffamily #1}}
\DeclareMathOperator{\ShHom}{\mathscr{H}\text{\kern -3pt {\calligra\large om}}\,}

\renewcommand{\a}{\alpha}

\newcommand{\D}{\mathcal{D}}

\newcommand{\bw}{\bigwedge}

\renewcommand{\ll}{\lambda}

\newcommand{\oo}{\otimes}

\renewcommand{\SS}{\mathbb{S}}

\let\uuu\u 
\renewcommand{\u}{\underline{u}}

\newcommand{\C}{\mathbb{C}}

\newcommand{\GL}{\operatorname{GL}}
\newcommand{\Hom}{\operatorname{Hom}}
\newcommand{\lie}{\mathfrak{g}}

\newcommand{\Spec}{\operatorname{Spec}}
\newcommand{\Sym}{\operatorname{Sym}}

\newcommand{\CC}{\operatorname{CC}}

\newcommand{\dom}{\operatorname{dom}}
\newcommand{\DR}{\operatorname{DR}}
\newcommand{\Eu}{\operatorname{Eu}}

\newcommand{\opmod}{\operatorname{mod}}
\newcommand{\op}{\operatorname}
\newcommand{\reg}{\operatorname{reg}}

\renewcommand{\skew}{\operatorname{skew}}
\newcommand{\symm}{\operatorname{symm}}
\newcommand{\opSS}{\operatorname{SS}}
\newcommand{\opS}{\operatorname{S}}

\newcommand{\bb}[1]{\mathbb{#1}}

\renewcommand{\rm}[1]{\textrm{#1}}
\newcommand{\mc}[1]{\mathcal{#1}}

\newcommand{\ol}[1]{\overline{#1}}

\def\lra{\longrightarrow}

\newtheorem{theorem}{Theorem}[section]
\newtheorem*{theorem*}{Theorem}
\newtheorem*{problem*}{Problem}
\newtheorem{lemma}[theorem]{Lemma}

\newtheorem{proposition}[theorem]{Proposition}
\newtheorem{corollary}[theorem]{Corollary}
\newtheorem*{corollary*}{Corollary}

\newtheorem*{main-thm*}{Main Theorem}
\newtheorem*{Euler-obstructions*}{Theorem on local Euler obstructions}
\newtheorem*{char-cycles*}{Theorem on characteristic cycles of simple $\D_X$-modules}
\newtheorem*{local-Euler*}{Theorem on intersection cohomology local Euler characteristics}
\newtheorem*{deRham*}{Theorem on de Rham cohomology of simple $\mc{D}_X$-modules}

\theoremstyle{definition}

\newtheorem*{definition*}{Definition}
\newtheorem{example}[theorem]{Example}

\theoremstyle{remark}
\newtheorem{remark}[theorem]{Remark}
\newtheorem*{remark*}{Remark}

\numberwithin{equation}{section}

\begin{document}

\title{Local Euler obstructions for determinantal varieties}

\author{Andr\'as C. L\H{o}rincz}
\address{Institut f\"ur Mathematik, Humboldt--Universit\"at zu Berlin, Berlin, Germany 12489}
\email{lorincza@hu-berlin.de}

\author{Claudiu Raicu}
\address{Department of Mathematics, University of Notre Dame, 255 Hurley, Notre Dame, IN 46556\newline
\indent Institute of Mathematics ``Simion Stoilow'' of the Romanian Academy}
\email{craicu@nd.edu}


\date{\today}


\begin{abstract} 
The goal of this note is to explain a derivation of the formulas for the local Euler obstructions of determinantal varieties of general, symmetric and skew-symmetric matrices, by studying the invariant de Rham complex and using character formulas for simple equivariant $\mc{D}$-modules. These calculations are then combined with standard arguments involving Kashiwara's local index formula and the description of characteristic cycles of simple equivariant $\mc{D}$-modules. The formulas are implicit in the work of Boe and Fu, and in the case of general matrices they have also been obtained recently by Gaffney--Grulha--Ruas, for skew-symmetric matrices by Promtapan and Rim\'anyi, and for all cases by Zhang.
\end{abstract}

\maketitle

\section{Introduction}\label{sec:intro}

Let $X$ be a smooth complex algebraic variety. To any closed subvariety $V\subseteq X$ one can associate a constructible function $\Eu_V$, whose value $\Eu_V(p)$ at a point $p$ is called the \defi{local Euler obstruction} of $V$ at $p$, and represents a measure of the singularity of $V$ at $p$. Local Euler obstructions were introduced by MacPherson using transcendental methods in \cite{macpherson}*{Section~3}, and were later given a purely algebraic description (\cite{gon-spr}*{Section~4.3}, \cites{sabbah,kennedy}). After giving a quick review of the theory of local Euler obstructions, our goal is to illustrate the theory by computing the local Euler obstructions for the determinantal varieties of general, symmetric and skew-symmetric matrices. The results for general matrices have appeared in \cite{gaffney-et-al}*{Theorem~2.17}, \cite{zhang}*{Theorem~3}, for skew-symmetric matrices in \cite{promtapan}*{Theorem~9.11} and \cite{promrim}*{Theorem~8.1}, and for all three cases in the recent work of Zhang \cite{zhang2}*{Section~6}. Here we follow a different approach using $\D$-module and representation-theoretic methods based on the invariant de Rham complex, thus providing a recipe for dealing with other representations with finitely many orbits as well (see Corollary \ref{cor:repfin}). The symmetric case is especially interesting because simple equivariant $\D$-modules supported on symmetric determinantal varieties can have reducible characteristic varieties, which is the main obstacle for computing the local Euler obstructions directly from Kazhdan--Lusztig theory (see for instance the proof of \cite{mihalcea-singh}*{Theorem~10.4}). Further, this case provides supporting evidence for the recent positivity conjecture of Mihalcea--Singh concerning the local Euler obstructions associated to the Schubert stratification of the Lagrangian Grassmannian \cite{mihalcea-singh}*{Conjecture~10.2}, which was verified in \cite{LeVan-Rai}.

Throughout this work, $X$ will denote one of the following affine spaces:
\begin{enumerate}
    \item $X^{m,n}\simeq\bb{C}^{m}\oo\bb{C}^n$ -- the space of $m\times n$ matrices, where $m\geq n$.
    \item $X^{n,\symm}\simeq\Sym^2(\bb{C}^n)$ -- the space of $n\times n$ symmetric matrices.
    \item $X^{n,\skew}\simeq\bigwedge^2(\bb{C}^n)$ -- the space of $n\times n$ skew-symmetric matrices.
\end{enumerate}

In each of the cases above, we will consider the rank stratification on $X$, denote the strata by $X_i$ and their closures by $V_i = \ol{X_i}$. More precisely:
\begin{enumerate}
    \item If $X=X^{m,n}$ then $X_i$ denotes the stratum of rank $i$ matrices, for $0\leq i\leq n$.
    \item If $X=X^{n,\symm}$ then $X_i$ denotes the stratum of rank $i$ symmetric matrices, for $0\leq i\leq n$.
    \item If $X=X^{n,\skew}$ then $X_i$ denotes the stratum of rank $2i$ skew-symmetric matrices, for $0\leq i\leq \lfloor n/2\rfloor$.
\end{enumerate}
The local Euler obstruction functions $\Eu_{V_i}$ are constant along each stratum (see for instance \cite{bra-sch}*{Corollaire~10.2}), so we can define
\begin{equation}\label{eq:def-eij}
e_{i,j} = \Eu_{V_j}(x_i)\text{ for any }x_i\in X_i.
\end{equation}
Since $X_i\subseteq V_j$ if and only if $i\leq j$, and $X_i$ is the non-singular locus of $V_i$, it follows from \cite{macpherson}*{Section~3} that
\[ e_{i,i} = 1\text{ and }e_{i,j} = 0\text{ for }i>j.\]
If we write $\mathds{1}_{X_i}$ for the indicator function of the stratum $X_i$, then \eqref{eq:def-eij} is equivalent to the expression
\[ \Eu_{V_j} = \sum_{i=0}^j e_{i,j} \cdot \mathds{1}_{X_i}.\]
The following theorem records the values of the local Euler obstructions for determinantal varieties (see also \cites{zhang, gaffney-et-al,promtapan,promrim,zhang2}) -- note that $e_{i,j}\geq 0$ in all cases, and see \cite{mihalcea-singh}*{Conjecture~10.2}.

\begin{Euler-obstructions*}
 The local Euler obstructions for determinantal varieties are given as follows.
 \begin{enumerate}
     \item If $X=X^{m\times n}$, $m\geq n$, then for all $0\leq i\leq j\leq n$ we have
     \[ e_{i,j} = {n-i\choose j-i}.\]
     \item If $X=X^{n,\symm}$ then
     \[ e_{i,j} = \begin{cases}
 0 & \text{ if }n-i\text{ is even and }n-j\text{ is odd}; \\
 \displaystyle{\lfloor\frac{n-i}{2}\rfloor \choose \lfloor\frac{j-i}{2}\rfloor} & \text{otherwise}.
 \end{cases}\]
     \item If $X=X^{n,\skew}$ and $m=\lfloor n/2\rfloor$ then for all $0\leq i\leq j\leq m$ we have
     \[ e_{i,j} = {m-i\choose j-i}.\]
 \end{enumerate}
\end{Euler-obstructions*}
Due to the fact that our matrix spaces are naturally identified with open subsets of classical Grassmannians (see the discussion below), the formulas in the theorem above are implicit in \cite{boe-fu}*{Section~6}. Similarly to \cite{boe-fu}, our strategy is to exploit the connection between $e_{i,j}$ and two closely related invariants of stratifications: characteristic cycles and intersection cohomology. Following \cite{ginsburg}*{p.~331}, we consider the following commutative diagram associated to $X$:

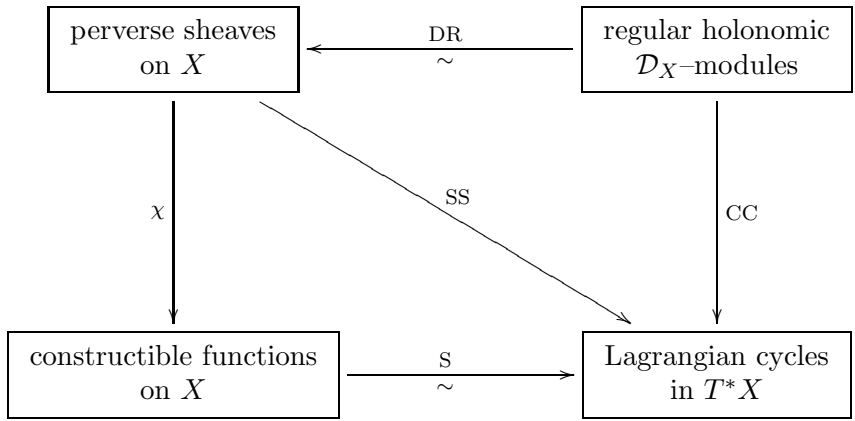
\[
\xymatrix@R+5pc{
{\boxed{\begin{array}{c}\text{perverse sheaves} \\ \text{on $X$} \end{array}}} \ar[d]_{\chi} \ar[drrr]^{\opSS} & & & {\boxed{\begin{array}{c}\text{regular holonomic} \\ \mc{D}_X\text{--modules} \end{array}}} \ar[lll]_{\DR}^{\sim} \ar[d]^{\CC} \\
{\boxed{\begin{array}{c}\text{constructible functions} \\ \text{on $X$} \end{array}}} \ar[rrr]^{\opS}_{\sim} & & & {\boxed{\begin{array}{c}\text{Lagrangian cycles} \\ \text{in }T^*X \end{array}}}  \\
}
\]

Here $\DR$ denotes the de Rham functor, $\CC(M)$ is the characteristic cycle of a $\D_X$-module $M$, or equivalently, the singular support ($\opSS$) of the associated perverse sheaf $\DR(M)$ (see \cite{htt}*{Section~2.2,~Chapter~7}, \cite{kas-sch}*{Chapters~V and IX}). The map $\chi$ associates to a perverse sheaf, or more generally to a constructible complex $\mc{F}^{\bullet}$, the local Euler characteristic function
\[ x \mapsto \chi_x(\mc{F}^{\bullet}):= \sum_i (-1)^i\cdot \dim H^i\left(\mc{F}^{\bullet}_x\right),\]
where $\mc{F}^{\bullet}_x$ denotes the stalk at the point $x$. A key fact in the diagram above is that $\opSS$ factors through $\chi$, and the local Euler obstruction function $\Eu_V$ is the unique constructible function (the reader can take this as the definition of $\Eu_V$) with the property that
\begin{equation}
 \opS(\Eu_V) = (-1)^{\dim V}\cdot T^*_V X,    
\end{equation}
where $T^*_V X$ denotes the closure of the conormal bundle to the regular locus $V^{\reg}\subseteq V$ (see \cite{sabbah}, \cite{kennedy}*{Lemma~4 and Section~4} and \cite{kas-sch}*{Section~9.7}).

Returning to our matrix spaces $X$ with their natural rank stratification, we let $IC^{\bullet}_{V_i}$ denote the intersection cohomology complex associated to $V_i$. Let $D_i$ be the simple, regular holonomic $\mc{D}_X$-module associated to $IC^{\bullet}_{V_i}$ through the Riemann--Hilbert correspondence \cite{htt}*{Theorem 7.2.5} (called the Brylinski--Kashiwara module of $V_i \subset X$, see \cite{bry-kash}*{Section 8}). We define the \defi{microlocal indices} $m_{i,j}$ via
\begin{equation}\label{eq:def-mij}
   \opSS(IC^{\bullet}_{V_j}) = \CC(D_j) = \sum_i m_{i,j} \cdot T^*_{V_i}X,
\end{equation}
We have $m_{i,j}=0$ for $i>j$, as well as the following.

\begin{char-cycles*}
The microlocal indices $m_{i,j}$ for determinantal varieties are as follows.
 \begin{enumerate}
     \item If $X=X^{m\times n}$ then $m_{i,i}=1$ and $m_{i,j}=0$ for $i\neq j$.
     \item If $X=X^{n,\symm}$ then $m_{i,i}=1$ for all $i$,
    \[m_{i-1,i} = 1\text{ if }n-i\text{ is odd},\]
     and $m_{i,j}=0$ otherwise.
     \item If $X=X^{n,\skew}$ then $m_{i,i}=1$ and $m_{i,j}=0$ for $i\neq j$.
 \end{enumerate}
\end{char-cycles*}

The description of characteristic cycles in the theorem above is noted in \cite{raicu-dmods}*{Remark~1.5} (see also \cite{lor-wal}*{Corollary~3.19 and Section~5} and \cite{braden-thesis}*{Section 3.4}). In general, establishing the irreducibility of characteristic cycles, or identifying the non-trivial components when they exist, is quite a difficult task! We have the following:
\begin{enumerate}
    \item The main result of \cite{BFL} gives the irreducibility of the characteristic cycles for the Schubert stratification of the (type A) Grassmannian. Since $X^{m,n}$ can be identified with the opposite big cell in the Grassmannian $\bb{G}(n,m+n)$ in such a way that the Schubert stratification refines the rank stratification (see \cite{lak-rag}*{Section~5.2.2}), part (1) of the theorem above is a special case of \cite{BFL}*{Theorem~0.1}. 
    \item The characteristic cycles for the (type C) Lagrangian Grassmannian $\bb{LG}(n,2n)$ are described combinatorially in \cite{boe-fu}*{Theorem~7.1D}, and they may be reducible. Since $X^{n,\symm}$ is the opposite big cell of $\bb{LG}(n,2n)$ (see \cite{lak-rag}*{Section~6.2.5}), one can then derive (with a little work) conclusion (2) of the theorem above.
    \item Finally, it follows from \cite{boe-fu}*{Theorem~7.1A} that the characteristic cycles for the (type D) orthogonal Grassmannian $\bb{OG}(n,2n)$ are irreducible. Since $X^{n,\skew}$ is the opposite big cell in $\bb{OG}(n,2n)$ (see \cite{lak-rag}*{Section~7.2.5}), conclusion (3) of the theorem above follows.
\end{enumerate}
For more non-trivial calculations of characteristic cycles, see \cite{kas-sai}*{Section~7}, \cites{braden-thesis,evens-mirkovic,timchenko}.

The final ingredient in our discussion is given by the intersection cohomology local Euler characteristics
\begin{equation}\label{eq:def-chi-ij}
    \chi_{i,j} = \chi_{x_i}(IC^{\bullet}_{V_j}) \text{ for any }x_i\in X_i.
\end{equation}
If we write $d_i = \dim(V_i)$ then we have
\[ \chi_{i,i} = (-1)^{d_i}\text{ and }\chi_{i,j}=0\text{ for }i>j.\]

\begin{local-Euler*}
 The intersection cohomology local Euler characteristics for determinantal varieties are computed as follows.
 \begin{enumerate}
     \item If $X=X^{m\times n}$, $m\geq n$, then for all $0\leq i\leq j\leq n$ we have
     \[ \chi_{i,j} = (-1)^{d_j}\cdot{n-i\choose j-i}.\]
     \item If $X=X^{n,\symm}$ then
        \[ \chi_{i,j} = (-1)^{d_j}\cdot\displaystyle{\lfloor \frac{n-i}{2}\rfloor + \epsilon \choose \lfloor\frac{j-i}{2}\rfloor},\text{ where }
     \epsilon = \begin{cases}
    1 & \text{if }(j-i)\text{ is even and }(n-i)\text{ is odd}; \\
    0 & \text{otherwise}.
    \end{cases}
    \]
     \item If $X=X^{n,\skew}$ and $m=\lfloor n/2\rfloor$ then for all $0\leq i\leq j\leq m$ we have
     \[ \chi_{i,j} = (-1)^{d_j}\cdot{m-i\choose j-i}.\]
 \end{enumerate}
 \end{local-Euler*}

For the Schubert stratification, the local intersection cohomology groups are computed as coefficients of Kazhdan--Lusztig polynomials \cite{KL-poincare}*{Theorem~4.3}, \cite{htt}*{Theorem~12.2.5}. Using the identification of our matrix spaces with opposite cells in an appropriate Grassmannian, it follows that $\chi_{i,j}$ is equal up to a sign with the value at $1$ of a corresponding Kazhdan--Lusztig polynomial. Explicit descriptions of these polynomials are given for $\bb{G}(n,m+n)$ in \cite{las-sch}*{Th\'eor\`eme~7.8}, for $\bb{LG}(n,2n)$ in \cite{boe}*{Theorem~3.13}, and for $\bb{OG}(n,2n)$ in \cite{boe}*{Theorem~4.1}. In the case of matrix spaces we will give an alternative derivation of the Kazhdan--Lusztig polynomials and the invariants $\chi_{i,j}$ based on a study of the invariant de Rham complex.

The connection between the three theorems listed above comes from the Kashiwara microlocal index formula \cite{kas-loc-ind}*{Section~2}, \cite{BDK}*{Th\'eor\`emes~1,~2}, \cite{ginsburg}*{Theorem~8.2}, which in our case can be phrased as the following identity:
\begin{equation}\label{eq:local-index} \chi_{i,j} = \sum_{k=i}^j (-1)^{d_k} \cdot e_{i,k} \cdot m_{k,j}.
\end{equation}
Equivalently, if we consider the upper-triangular matrices
\[ \mc{X} = (\chi_{i,j}),\quad \mc{E} = (e_{i,j}),\quad \mc{M} = \left((-1)^{d_i} m_{i,j}\right) \]
then we have the identity
\begin{equation}\label{eq:X=EM}
\mc{X} = \mc{E} \cdot \mc{M}.
\end{equation}
\begin{example}\label{ex:2x2-symm}
 Consider the space $X^{2,\symm}$ of $2\times 2$ symmetric matrices. We have $d_0=0$, $d_1=2$, $d_2=3$, and
 \[ \mc{X} = \begin{bmatrix} 1 & 1 & -1 \\ 0 & 1 & -1 \\ 0 & 0 & -1 \end{bmatrix},\quad
  \mc{E} = \begin{bmatrix} 1 & 0 & 1 \\ 0 & 1 & 1 \\ 0 & 0 & 1 \end{bmatrix},\quad
 \mc{M} = \begin{bmatrix} 1 & 1 & 0 \\ 0 & 1 & 0 \\ 0 & 0 & -1 \end{bmatrix}.
 \]
 We single out the local Euler obstruction $e_{0,1} = \Eu_{V_1}(0) = 0$, where $V_1$ is a quadric cone (defined by the vanishing of the determinant of the $2\times 2$ symmetric matrix), and $0$ is the vertex of $V_1$. It was noted in \cite{macpherson}*{Section~3} that if $V$ is the cone with vertex $0$ over a non-singular plane curve of degree $d$, then $\Eu_V(0) = 2d-d^2$, so our example coincides with the special case $d=2$.
\end{example}

Using the inductive structure of determinantal varieties, it is not difficult to reduce the calculation of the invariants $\chi_{i,j}$ to that of $\chi_{0,j-i}$ for a smaller matrix space. For this reason, it is enough to compute the local Euler characteristic of the stalk at $0$ for each $IC_{V_j}^{\bullet}$, which in turn gets identified via the $\bb{C}^*$-action with the global intersection cohomology Euler characteristic of $V_j$. The intersection cohomology of $V_j$ agrees (up to a shift) with the de Rham cohomology of the associated simple $\mc{D}$-module $D_j$, and it is encoded by a Kazhdan--Lusztig polynomial. We show that it can be computed by restricting to the invariant de Rham complex (Corollary~\ref{cor:invderham}), which we determine explicitly and observe that it has no non-zero differentials.

\begin{deRham*} If we write $h^i_{dR}(D_p)$ for the dimension of $H^i_{dR}(D_p)$, and $\Omega^i_X$ for the module of $i$-differentials on $X$, then we have.
 \begin{enumerate}
     \item If $X=X^{m\times n}$, $m\geq n$, and $G=\GL_m\times\GL_n$, then for all $0\leq p\leq n$ we have
     \[\sum_{i\geq 0} \dim(\Omega^i_X\oo D_p)^{G} \cdot q^i = \sum_{i\geq 0} h^i_{dR}(D_p) \cdot q^i = {n \choose p}_{q^2} \cdot q^{(m-p)\cdot(n-p)}.\]
     \item If $X=X^{n,\symm}$ and $G=\GL_n$, then for all $0\leq p\leq n$ we have
     \[\sum_{i\geq 0} \dim(\Omega^i_X\oo D_p)^{G} \cdot q^i = \sum_{i\geq 0} h^i_{dR}(D_p) \cdot q^i = {\lfloor\frac{n}{2}\rfloor+\epsilon \choose \lfloor\frac{p}{2}\rfloor}_{q^4} \cdot q^{{n-p+1\choose 2}},\]
     where $\epsilon=1$ if $p$ is even and $n$ is odd, and $\epsilon=0$ otherwise.
     \item If $X=X^{n,\skew}$, $G=\GL_n$, and $m=\lfloor n/2\rfloor$, then for all $0\leq p\leq m$ we have
     \[\sum_{i\geq 0} \dim(\Omega^i_X\oo D_p)^{G} \cdot q^i = \sum_{i\geq 0} h^i_{dR}(D_p) \cdot q^i = {m \choose p}_{q^4} \cdot q^{{n\choose 2}-p(2n-2p-1)}.\]
 \end{enumerate}
\end{deRham*}

\noindent{\bf Organization.} We have no additional input regarding the calculation of the multiplicities $m_{i,j}$, for which we refer the reader to the cited sources. In Section~\ref{sec:deRham} we discuss equivariant $\mc{D}$-modules and the invariant de Rham complex, and in Section~\ref{sec:prelim} we recall some basics on representations of the general linear group. We then discuss the calculation of Kazhdan--Lusztig polynomials via the invariant de Rham complex, along with the invariants $\chi_{i,j}$ and $e_{i,j}$: the case $X=X^{m,n}$ is treated in Section~\ref{sec:general-matrices}, the case $X=X^{n,\symm}$ in Section~\ref{sec:symmetric}, and the case $X=X^{n,\skew}$ in Section~\ref{sec:skew}.

\section{Equivariant $\mc{D}$-modules and the invariant de Rham complex}
\label{sec:deRham}

Let $X$ be an irreducible smooth complex affine variety of dimension $d$ and let $\mc{D}=\mc{D}_X$ denote the sheaf of differential operators on $X$. Throughout we make an identification between quasi-coherent $\mc{O}_X$-modules and their global sections. For a $\mc{D}$-module $M$, we consider the (algebraic) de Rham complex
\begin{equation}\label{eq:drcomp}
DR(M): \qquad 0 \lra M \lra \Omega^1_X \oo_{\mc{O}_X} M \lra \cdots \lra \Omega^{d}_X \oo_{\mc{O}_X} M \lra 0,
\end{equation}
where $\Omega^i_X$ is the module of $i$-differential forms and is placed in cohomological degree $i$. Writing $\pi_+(M)$ for the $\D$-module-theoretic derived integration (pushforward), we have an identification of the de Rham cohomology groups $H^i_{dR}(M)\simeq H^{i-d}(\pi_+ (M))$ for all $i$. In particular, if $M$ is holonomic then each $H^i_{dR}(M)$ is finite-dimensional \cite{htt}*{Theorem 3.2.3}.

The corresponding analytic de Rham complex plays a fundamental role in the Riemann--Hilbert correspondence \cite{htt}*{Theorem 7.2.5}. In the special case when $M=\mc{O}_X$ is the structure sheaf, the celebrated comparison theorem of Grothendieck \cite{grothendieck} implies that the space $H^i_{dR}(\mc{O}_X)$ agrees with the (singular) cohomology group $H^i(X,\C)$. More generally, for an irreducible closed subvariety $Y\subset X$ there is a corresponding simple $\mc{D}$-module $\mc{L}(Y,X)$ (called the Brylinski--Kashiwara module \cite{bry-kash}*{Section 8}) whose associated de Rham complex is, up to a shift, the middle perversity intersection cohomology complex $IC_Y^{\bullet}$. In particular, the de Rham cohomology groups of $\mc{L}(Y,X)$ agree with the intersection cohomology groups of $Y$ (see for instance \cite{htt}*{Theorem 7.1.1}). As we explain in Section~\ref{sec:invderham}, in the presence of a (connected, reductive) group action it suffices to work with the invariant de Rham complex in order to compute cohomology, the advantage being that this complex is in general much smaller than \eqref{eq:drcomp}.

\subsection{Equivariant $\D$-modules}\label{subsec:equivd}

In this section we provide some background on equivariant $\D$-modules (for more details, see \cite{lor-wal}*{Section~2.1}). We assume that $G$ is a connected algebraic group acting on $X$, and we say that $M$ is a (strongly) $G$-\defi{equivariant} $\D$-module if there exists a $\D_{G\times X}$-isomorphism
\begin{equation}\label{eq:tau-iso-pullbacksM}
\tau\colon p^*M \rightarrow m^*M,
\end{equation}
where $p$ and $m$ are the projection and multiplication maps
\[
p\colon G\times X\to X,\qquad\qquad m\colon G\times X\to X
\]
respectively, and $\tau$ satisfies the usual compatibility conditions on $G\times G\times X$ \cite{htt}*{Definition 11.5.2}.

We call a (possibly infinite-dimensional) vector space $V$ a \defi{rational $G$-module} if $V$ is equipped with a linear action of $G$, such that every $v\in V$ is contained in a finite-dimensional $G$-stable subspace $W$, where the $G$-action on $W$ is given by a morphism $G \to \GL(W)$ of algebraic varieties. If we let $\lie$ denote the Lie algebra of $G$, then by differentiating the action of $G$ on $X$ we get a map $\lie \to \D_X$. Equivariance of a $\D$-module $M$ amounts to $M$ having a rational $G$-module structure such that differentiating the action of $G$ on $M$ coincides with the action of $\lie$ induced from $\lie \to \D_X$. In particular, if a map $\tau$ as in \eqref{eq:tau-iso-pullbacksM} exists then it must be unique, hence the notion of a $\D$-module being equivariant should be thought of as a property the module, rather than as additional data. Therefore, coherent equivariant $\D$-modules form a full subcategory $\op{mod}_G(\D_X)$ of the category $\op{mod}(\D_X)$ of coherent $\D_X$-modules. Moreover, if $f$ is any $G$-equivariant map $f:X \to X'$ between smooth $G$-varieties $X,X'$, then the $\D$-module-theoretic direct image $f_+$ preserves equivariance.

\subsection{Invariant de Rham complex}\label{sec:invderham}

One of the most basic results in algebraic topology is that (co)homology groups of a topological space are homotopy-invariant. Moreover, homotopic maps between two spaces induce identical maps on the level of cohomology. An immediate consequence of this is that the action of a connected group $G$ on $X$ induces the trivial action on the cohomology groups $H^i(X,\C)$. This can be viewed as a special case (with $M=\mc{O}_X$) of the following result.

\begin{lemma}\label{lem:trivact}
Let $M$ be a $G$-equivariant $\D_X$-module, and assume that $G$ is connected. The induced action of $G$ on $H^i_{dR}(M)$ is trivial for all $i\geq 0$.
\end{lemma}

\begin{proof}
Let $\pi: X\to \{pt\}$ be the map to the point, which is $G$-equivariant. Since the pushforward preserves equivariance (Section \ref{subsec:equivd}), the spaces $H^{k}(\pi_+ (M))$ are equivariant $\D_{pt}$-modules, for all $k\in\bb{Z}$. The map $\lie \to \D_{pt}=\C$ is zero, hence $\lie$, and consequently $G$ acts trivially on any equivariant $\D_{pt}$-module.  Since we have $H^i_{dR}(M)=H^{i-d}(\pi_+ (M))$ for all $i$, the desired conclusion follows.
\end{proof}

It is clear that the differentials in the de Rham complex (\ref{eq:drcomp}) are $G$-equivariant. The following result shows that in order to compute de Rham cohomology, it is enough to consider the $G$-invariant part of this complex.

\begin{corollary}\label{cor:invderham}
With the notation as in Lemma \ref{lem:trivact}, assume in addition that $G$ is reductive. For each $i$, we have that $H^i_{dR}(M)$ is isomorphic to the $i$th cohomology of the invariant de Rham complex
 \begin{equation}\label{eq:invdrcomp}
DR(M)^G:  \qquad 0 \lra M^G \lra (\Omega^1_X \oo M)^G \lra \cdots \lra (\Omega^{d}_X \oo M)^G \lra 0,
\end{equation}
\end{corollary}

\begin{proof}
By Lemma \ref{lem:trivact} we have $H^i_{dR}(M)=H^i_{dR}(M)^G=H^i(DR(M))^G=H^i(DR(M)^G)$, where the last equality follows from the fact that taking $G$-invariants is an exact functor when $G$ is reductive.
\end{proof}

Next, we illustrate the effectiveness of calculating with the invariant de Rham complex.

\begin{proposition}\label{prop:invdrfinite}
Suppose that $G$ is a connected, reductive group, acting on $X$ with finitely many orbits. For any $M\in \op{mod}_G(\D_X)$, each term in the complex $DR(M)^G$ from (\ref{eq:invdrcomp}) is finite-dimensional. 
\end{proposition}

\begin{proof}
We write $S=\mc{O}_X$, and we show that if $A$ is any finitely-generated $S$-module with a compatible rational $G$-module structure (such as $\Omega^i_X$), then $(A\oo_S M)^G$ is finite-dimensional. By choosing a $G$-equivariant finite dimensional space of generators $V$ of $A$, we get a surjective $G$-equivariant map of $S$-modules
\[S\oo_\C V \twoheadrightarrow A.\] 
This in turn induces a surjective map $(V\oo_\C M)^G\twoheadrightarrow (A\oo_S M)^G$. Hence, it is enough to prove that $(V\oo M)^G$ is finite dimensional. 

For an irreducible representation $Z$ of $G$, let $[V:Z]$ (resp. $[M:Z]$) denote the multiplicity of $Z$ in the $G$-decomposition of $V$ (resp. $M$). Let $\mc{I}$ denote the finite set of isomorphism classes of irreducible $G$-modules $W$ that appear in the decomposition of $V$ (that is, $[V:W]\geq 1$). We have 
\[\dim(V\oo M)^G = \sum_{W\in \mc{I}} [V:W]\cdot [M:W^\vee],\]
which is finite, since $M$ is a multiplicity-finite $G$-module by \cite{lor-wal}*{Proposition 3.14}.
\end{proof}

\begin{remark}
Under the assumptions of Proposition \ref{prop:invdrfinite}, any $M\in \op{mod}_G(\D_X)$ is in fact (regular) holonomic by \cite{htt}*{Theorem 11.6.1}, hence the finiteness of $h_{dR}^i(M)$ follows readily. Nonetheless, the conclusion of Proposition~\ref{prop:invdrfinite} is stronger, as the finiteness occurs already at the level of the complex $DR(M)^G$.
\end{remark}

We thus obtain a representation-theoretic approach for computing the intersection cohomology local Euler characteristic of strata at the most singular point.

\begin{corollary}\label{cor:repfin}
Assume that $X$ is a representation of a connected, reductive group $G$ with finitely many orbits. For an orbit $O \subset X$, let $D_O$ denote the corresponding Brylinski--Kashiwara $\D_X$-module. The intersection cohomology local Euler characteristic of $\ol{O}$ at $0$ is given by
\[\chi_{0, \ol{O}} \, = \, \sum_{i=0}^d \, (-1)^{d-i} \cdot \dim_{\C} \, (D_O \oo_{\C} \bigwedge^i X^{\vee})^G.\]
\end{corollary}

\begin{proof}
From Corollary \ref{cor:invderham} and Proposition \ref{prop:invdrfinite} we get
\[\sum_{i=0}^d (-1)^i \cdot h^i_{dR}(D_O) = \sum_{i=0}^d (-1)^i \cdot \dim_{\C} (D_O \oo_{\C} \bigwedge^i X^{\vee})^G.\]
Furthermore, from the discussion above we have $h^{i-d}\left(IC^{\bullet}_{\ol{O}}\right) = h^i_{dR}(D_O)\text{ for all }i$. Since $X$ has finitely many orbits, $O$ is stable under the action of $\C^*$. Then global intersection cohomology agrees with the stalk intersection cohomology at $0$ (see \cite{htt}*{(13.2.40)}), hence the conclusion.
\end{proof}

\section{Preliminaries on representations of the general linear group}\label{sec:prelim}

In this section we recall some basic facts and notation regarding partitions and the representation theory of $\GL_n(\C)$. We write $\bb{Z}^n_{dom}$ for the set of \defi{dominant weights} in $\bb{Z}^n$:
\[\bb{Z}^n_{dom} = \{\ll=(\ll_1,\cdots,\ll_n)\in\bb{Z}^n:\ll_1\geq\ll_2\geq\cdots\geq\ll_n\}.\]
When each $\ll_i\geq 0$, we identify $\ll$ with a \defi{partition} with (at most) $n$ parts, and write $\ll\in\bb{N}^n_{dom}$. We let $|\ll|:=\ll_1+\cdots+\ll_n$ denote the \defi{size} of the partition $\lambda$, and write $\lambda\vdash k$ when $|\lambda|=k$. We identify a partition $\lambda$ with its \defi{Young diagram}, consisting of left-justified rows of boxes, where row $i$ consists of $\lambda_i$ boxes. For example, $\lambda=(5,3,3,2)$ has Young diagram
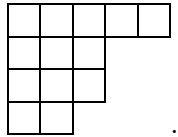
\[\yng(5,3,3,2).\]

The \defi{Durfee size} of $\ll$ is the largest $s$ with the property $\ll_s\geq s$, and it corresponds to the largest $s\times s$ square contained in the Young diagram of $\lambda$. We have for instance that $\lambda=(5,3,3,2)$ has Durfee size $3$. We write $\lambda'$ for the \defi{conjugate} partition of $\lambda$, where $\ll'_i$ counts the number of parts $\ll_j$ with $\ll_j\geq i$. We note that the Young diagram of $\lambda'$ is obtained by transposing the one of $\lambda$, and that in particular, the Durfee size of $\lambda'$ is equal to that of $\lambda$. For example, we have
\[(5,3,3,2)' = (4,4,3,1,1).\] 
We partially order $\bb{Z}^n_{dom}$ (and $\bb{N}^n_{dom}$) by declaring that $\ll\geq \mu$ if $\ll_i\geq\mu_i$ for all $i=1,\cdots,n$. 
If $a\geq 0$ then we write $a\times b$ or $(b^a)$ for the sequence $(b,b,\cdots,b)$ where $b$ is repeated $a$ times. With this notation, the Durfee size of $\lambda\in\bb{N}^n_{dom}$ is the largest $s$ for which $\lambda\geq(s^s)$.

If $F$ is a vector space with $\dim(F)=n$ and $\ll\in\bb{Z}^n_{dom}$, then we write $\SS_{\ll}F$ for the corresponding irreducible representation of $\GL(F)$ (or \defi{Schur functor}). Our conventions are such that if $\ll=(k,0,\cdots,0)$ then we have $\SS_{\ll}F = \Sym^k F$, while for $\ll=(1^r,0^{n-r})$ we have $\bb{S}_{\ll}F=\bw^r F$. For a weight $\ll\in\bb{Z}^n$, we define its \defi{dual} to be
\[ \ll^{\vee} = (-\ll_n,-\ll_{n-1},\cdots,-\ll_1),\]
and we have an isomorphism
\begin{equation}\label{eq:Slam-dual}
    \SS_{\ll}(F^{\vee}) \simeq \SS_{\ll^{\vee}}(F)\text{ for all }\ll\in\bb{Z}^n_{dom}.
\end{equation}
Moreover, we have by Schur's lemma that if we let $G=\GL(F)$ then
\begin{equation}\label{eq:Schur-lemma}
\left(\SS_{\ll}F^{\vee} \oo \SS_{\mu}F\right)^G = \Hom_G\left(\SS_{\ll}F,\SS_{\mu}F\right) = \begin{cases}
 \bb{C} & \text{if }\ll=\mu;\\
 0 & \text{otherwise}.
\end{cases} 
\end{equation}

When $m\geq n$, we will think of $\bb{N}^n_{dom}$ as a subset of $\bb{N}^m_{dom}$ by adding trailing zeroes, and think of the union of all $\bb{N}^n_{dom}$ as the set of partitions. For a partition $\ll$ we can then think of $\bb{S}_{\ll}$ as a functor of finite dimensional vector spaces (of any dimension), with
\begin{equation}\label{eq:SllF=0}
\SS_{\ll}F = 0\text{ if }\lambda_i\neq 0\text{ for some }i>\dim(F).
\end{equation}
Note that when $\ll=(1^r)$, the formula above states the familiar fact that $\bw^r F = 0$ when $r>\dim(F)$.

\subsection{Plethysm formulas}
We next record some fundamental plethysm formulas that will be used in analyzing the invariant de Rham complex for matrix spaces. Suppose first that we have two finite dimensional vector spaces $F_1,F_2$. The structure of the exterior powers on $F_1\oo F_2$ as representations of $G=\GL(F_1)\times\GL(F_2)$ is governed by Cauchy's formula \cite{weyman}*{Corollary~2.3.3}:
\begin{equation}\label{eq:extcauchy}
\bw^i (F_1\oo F_2) = \bigoplus_{\ll \vdash i} \SS_{\ll} F_1 \otimes \SS_{\ll'} F_2, \quad \mbox { for all } i\geq 0.
\end{equation}
If we let $m=\dim(F_1)$ and $n=\dim(F_2)$, then using \eqref{eq:SllF=0} we can restrict \eqref{eq:extcauchy} to those partitions $\lambda$ with $\lambda\leq(n^m)$: if $\lambda_{m+1}\neq 0$ then $\SS_{\ll}F_1=0$, while if $\lambda_1>n$ then $\SS_{\ll'}F_2=0$.

The next plethysm formulas concern a single vector space $F$. We have by \cite{weyman}*{Proposition~2.3.9(a)}
\begin{equation}\label{eq:extsymm}
\bw^i\left(\Sym^2 F\right)=\bigoplus_{\ll \in \mc{Y}(2i)} \SS_{\ll} F, \quad \mbox { for all } i\geq 0,
\end{equation}
where $\mc{Y}(2i)$ denotes the set of partitions $\ll \vdash 2i$ of some Durfee size $r\geq 0$ for which there is a partition $\a \leq (n-r)^r$ such that 
\begin{equation}\label{eq:def-Y}
\ll_j = \begin{cases}
r+1+\a_j & \mbox{for }j=1,\cdots,r,\\
\a'_{j-r} & \mbox{for }j=r+1,\cdots,n,
\end{cases}
\quad\text{ or pictorially}\quad
\ll: \hspace{0.18in} \begin{aligned}
\begin{tikzpicture}
\draw (0,1.5)--(0,0)--(0.5,0)--(0.5,0.5)--(1.5,0.5)--(1.5,1)--(2,1)--(2,1.5);
\node at (0.8,1) {$\a'$};
\draw [thick] (0,1.5)--(0,3.5)--(2,3.5)--(2,1.5)--(0,1.5);
\node at (1,2.5) {$r\times r$};
\node at (2.25,2.87) {$r$};
\node at (2.25,2.5) {$\times$};
\node at (2.25,2.13) {$1$};
\draw (2,1.5)--(2.5,1.5)--(2.5,3.5)--(2,3.5);
\draw(2.5,1.5)--(3,1.5)--(3,2)--(3.5,2)--(3.5,3)--(4,3)--(4,3.5)--(2.5,3.5);
\node at (3,2.6) {$\a$};
\end{tikzpicture}
\end{aligned}
\end{equation}

Similarly, the exterior powers of $\bw^2 F$ are given by \cite{weyman}*{Proposition~2.3.9(b)}:
\begin{equation}\label{eq:extskew}
\bw^i\left(\bw^2 F\right)=\bigoplus_{\ll \in \mc{Z}(2i)} \SS_{\ll} F, \quad \mbox { for all } i\geq 0,
\end{equation}
where $\mc{Z}(2i)$ denotes the set of partitions $\ll \vdash 2i$ of some Durfee size $r\geq 0$ for which there is a partition $\a \leq (n-r-1)^r$ such that: \vspace{-0.15in}
\begin{equation}\label{eq:def-Z}
\ll_j = \begin{cases}
r+\a_j & \mbox{for }j=1,\cdots,r,\\
r & \mbox{for }j=r+1,\\
\a'_{j-r-1} & \mbox{for }j=r+2,\cdots,n,
\end{cases}
\quad\text{ or pictorially}\quad
\ll: \hspace{0.18in} \begin{aligned}
\begin{tikzpicture}
\draw (0,1.5)--(0,0)--(0.5,0)--(0.5,0.5)--(1.5,0.5)--(1.5,1)--(2,1)--(2,1.5);
\node at (0.80,1) {$\a'$};
\node at (1,1.73) {$1 \times r$};
\draw [thick] (0,2)--(0,4)--(2,4)--(2,2)--(0,2);
\draw (0,1.5)--(0,2)--(2,2)--(2,1.5)--(0,1.5);
\node at (1,3) {$r\times r$};
\draw(2,2)--(2.5,2)--(2.5,2.5)--(3,2.5)--(3,3.5)--(3.5,3.5)--(3.5,4)--(2,4);
\node at (2.5,3.1) {$\a$};
\end{tikzpicture}
\end{aligned}
\vspace{-0.1in}
\end{equation}

\subsection{$q$-binomial coefficients}\label{subsec:Gaussian}
For $a\geq b\geq 0$ we define the \defi{Gaussian (or $q$-)binomial coefficient} ${a\choose b}_q$ to be the polynomial in $\bb{Z}[q]$ defined as
\[{a\choose b}_q = \frac{(1-q^a)\cdot(1-q^{a-1})\cdots (1-q^{a-b+1})}{(1-q^b)\cdot(1-q^{b-1})\cdots (1-q)}.\]
These polynomials are generalizations of the usual binomial coefficients, satisfying the relations
\begin{equation}\label{eq:qbin-is-bin}
{a\choose b}_q = {a\choose a-b}_q,\quad{a\choose a}_q={a\choose 0}_q=1,\mbox{ and }{a\choose b}_1 = {a\choose b},
\end{equation}
and the recursion
\begin{equation}\label{eq:qbin-recursion}
{a\choose b}_q = {a-1\choose b-1}_q + {a-1\choose b}_q \cdot q^b.
\end{equation}
One significance of the $q$-binomial coefficients is that ${a\choose b}_{q^2}$ describes the Poincar\'e polynomial of the Grassmannian $\bb{G}(b,a)$ of $b$-dimensional subspaces of $\bb{C}^a$. As such, the coefficient of $q^j$ in ${a\choose b}_q$ computes the number of Schubert classes of (co)dimension $j$, or equivalently the number of partitions $\ll$ of size $j$ contained inside the rectangular partition $(a-b)\times b$. We get
\begin{equation}\label{eq:qbin-genfun}
{a\choose b}_q = \sum_{\ll\leq (b^{a-b})} q^{|\ll|}.
\end{equation}

\section{General matrices}
\label{sec:general-matrices}

In this section we consider $X=X^{m,n}$, the affine space of complex $m\times n$ matrices, and we assume that $m\geq n\geq 1$. For a coordinate independent notation, we consider complex vector spaces $F_1,F_2$, $\dim(F_1)=m$, $\dim(F_2)=n$, we let $S = \Sym(F_1\oo F_2)$ and identify $X$ with $\Spec(S)$. We write $G= \GL(F_1) \times \GL(F_2)$ and consider its natural action on $X$. The rank stratification of $X$ agrees with the orbit stratification relative to the $G$-action. The dimensions of the strata are given by
\[ d_p = \dim(X_p) = p\cdot(m+n-p) \text{ for }p=0,\cdots,n.\]

The module of differential forms on $X$ is naturally isomorphic to
\[ \Omega^1_X = F_1 \oo F_2 \oo S,\]
and using (\ref{eq:extcauchy}) with $\mu=\ll'$, we have for $0\leq i\leq mn$ that
\begin{equation}\label{eq:omegadet}
\Omega^i_X = \bw^i\Omega^1_X = \bigoplus_{\mu\vdash i} \bb{S}_{\mu'}F_1 \oo \bb{S}_{\mu}F_2 \oo S.
\end{equation}

Each stratum gives rise to a simple object in $\opmod_{G}(\D_X)$, the Brylinski--Kashiwara module $D_p := \mc{L}(V_p, X)$. As a $G$-representation, $D_p$ has a decomposition into a direct sum of irreducible representations given by (see \cite{raicu-dmods}*{Section~5}, and also \cite{raicu-weyman}*{Theorem~6.1}, \cite{raicu-weyman-loccoh}*{Main Theorem(1)}, \cite{raicu-survey}*{Theorem~5.1})
\begin{equation}\label{eq:decomp-Dp}
 D_p = \bigoplus_{\ll \in \mc{A}(p)}\bb{S}_{\ll(n-p)} F_1^{\vee} \oo \bb{S}_{\ll} F_2^{\vee},
\end{equation}
where 
\[ \mc{A}(p) = \{\ll\in\bb{Z}^n_{\dom} : \ll_{n-p}\geq m-p\mbox{ and }\ll_{n-p+1}\leq n-p\},\]
and where for $0\leq s\leq n$ we denote
\[\ll(s) = (\ll_1-(m-n),\cdots,\ll_s-(m-n),s^{m-n},\ll_{s+1},,\cdots,\ll_n)\in\bb{Z}^m.\]
Notice that for $\ll\in\mc{A}(p)$ we have that $\ll(n-p)$ is dominant, so the corresponding Schur functor in \eqref{eq:decomp-Dp} is defined. We are now ready to analyze the invariant de Rham complex (\ref{eq:invdrcomp}) and compute de Rham cohomology for the simples $D_p$.

\begin{theorem}\label{thm:deRham-Dp} For $0\leq p \leq n\leq m$, we have
\begin{equation}\label{eq:inv-deRham-Dp-general}
 \sum_{i\geq 0} \dim(\Omega^i_X\oo D_p)^{G} \cdot q^i = {n \choose p}_{q^2} \cdot q^{(m-p)\cdot(n-p)}.
\end{equation}
In particular, the differentials in the invariant de Rham complex $DR(D_p)^{G}$ are identically zero, and
\begin{equation}\label{eq:deRham-Dp}
\sum_{i\geq 0} h^i_{dR}(D_p) \cdot q^i =  {n \choose p}_{q^2} \cdot q^{(m-p)\cdot(n-p)}.
\end{equation}
\end{theorem}

\begin{proof}
It follows from \eqref{eq:Schur-lemma}, \eqref{eq:omegadet}, and \eqref{eq:decomp-Dp} that $\dim(\Omega^i_X\oo D_p)^{G}$ is equal to the number of partitions $\mu\vdash i$ satisfying
\[\mu\in\mc{A}(p),\quad\text{and}\quad \mu'=\mu(n-p).\]
Note that the conditions $\mu_{n-p}\geq m-p\geq n-p$ and $\mu_{n-p+1}\leq n-p$ imply that $\mu$ has Durfee size $(n-p)$, while the condition $\mu'=\mu(n-p)$ is equivalent to
\begin{equation}\label{eq:mu-vs-muprime}
\mu'_j = \begin{cases}
 \mu_j - (m-n) & \mbox{for }j=1,\cdots,n-p, \\
  n-p & \mbox{for }j=n-p+1,\cdots,m-p,\\
 \mu_{j-(m-n)} & \mbox{for }j = m-p+1,\cdots m.\\
 \end{cases}
\end{equation}
Pictorially, the Young diagram of the partition $\mu'$ must have the following shape:
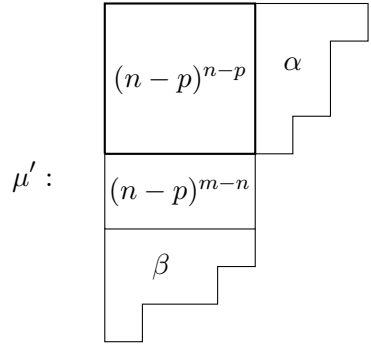
\[ \hspace{-0.3in} \mu': \hspace{0.2in}
\begin{aligned}
\begin{tikzpicture}
\draw (0,1.5)--(0,0)--(0.5,0)--(0.5,0.5)--(1.5,0.5)--(1.5,1)--(2,1)--(2,1.5);
\node at (0.75,1) {$\beta$};
\draw [thick] (0,2.5)--(0,4.5)--(2,4.5)--(2,2.5)--(0,2.5);
\draw (0,1.5)--(0,2.5)--(2,2.5)--(2,1.5)--(0,1.5);
\node at (1,2) {$(n-p)^{m-n}$};
\node at (1,3.5) {$(n-p)^{n-p}$};
\draw(2,2.5)--(2.5,2.5)--(2.5,3)--(3,3)--(3,4)--(3.5,4)--(3.5,4.5)--(2,4.5);
\node at (2.5,3.7) {$\a$};
\end{tikzpicture}
\end{aligned}\]
where $\alpha\in\bb{N}^{n-p}_{dom}$ is given by
\begin{equation}\label{eq:alpha-genl} 
\alpha_j = \mu'_j - (n-p) = \mu_j - (m-p)\text{ for }j=1,\cdots,n-p,
\end{equation}
and $\beta\in\bb{N}^p_{dom}$ satisfies (since the columns of $\mu'$ are the rows of $\mu$)
\begin{equation}\label{eq:beta-genl}
\beta'_j = \mu_j - (m-p) \text{ for }j=1,\cdots,n-p.
\end{equation}
Comparing \eqref{eq:mu-vs-muprime}, \eqref{eq:alpha-genl}, \eqref{eq:beta-genl}, we see that $\beta'=\alpha$, and therefore $\mu$ uniquely corresponds to $\alpha\leq \left(p^{n-p}\right)$. Since $\mu \vdash i$, we must have $i = 2|\a| + (m-p)\cdot(n-p)$. We obtain
\[
\sum_{i\geq 0} \dim(\Omega^i_X\oo D_p)^{G} \cdot q^i = 
\sum_{\alpha\leq(p^{n-p})} q^{2|\a| + (m-p)\cdot(n-p)} = {n \choose p}_{q^2} \cdot q^{(m-p)\cdot(n-p)},
\]
where the last equality follows from (\ref{eq:qbin-genfun}) and proves \eqref{eq:inv-deRham-Dp-general}.

Since the coefficient of $q^i$ is non-zero only for $i\equiv (m-p)\cdot(n-p)\mod 2$, it follows that every other term in the invariant de Rham complex $DR(D_p)^G$ is zero. This implies that the differentials are identically zero, and therefore
\[ h^i_{dR}(D_p) = \dim(\Omega^i_X\oo D_p)^{G}\text{ for all }i,\]
from which \eqref{eq:deRham-Dp} follows.
\end{proof}

As a corollary, we derive the following well-known formula for the intersection cohomology of determinantal varieties.

\begin{corollary}\label{cor:IC-general-mat}
 We have
 \[ \sum_{i\in\bb{Z}} h^i\left(IC^{\bullet}_{V_p}\right) \cdot q^i = q^{-d_p} \cdot {n\choose p}_{q^2}.\]
\end{corollary}

\begin{proof}
 Since $\dim(X)=mn$, we have by the discussion in Section~\ref{sec:deRham} that 
 \[h^{i-mn}\left(IC^{\bullet}_{V_p}\right) = h^i_{dR}(D_p)\text{ for all }i,\]
 hence the conclusion follows from the fact that $(m-p)\cdot(n-p)-mn=-d_p$.
\end{proof}

We now explain the calculation of the invariants $\chi_{i,j}$ and $e_{i,j}$ discussed in the Introduction.

\begin{corollary}\label{cor:chi-eij-general}
 For $0\leq i\leq j\leq n$ we have that
 \[\chi_{i,j} = (-1)^{d_j}\cdot {n-i \choose j-i} \text{ and }e_{i,j} = {n-i \choose j-i}.\]
\end{corollary}

\begin{proof}
 Recall that the microlocal indices $m_{i,j}$ are non-zero only when $i=j$, in which case $m_{i,i}=1$. It follows from \eqref{eq:local-index} that
 \[ \chi_{i,j} = (-1)^{d_j}\cdot e_{i,j}\text{ for all }i,j,\]
 so the desired formula for $\chi_{i,j}$ is equivalent to that for $e_{i,j}$.
 
 Consider first the case $i=0$, and note that $h^i(IC^{\bullet}_{V_j}) = h^i(IC^{\bullet}_{V_j,0})$: since $V_j$ is invariant under the scaling action of $\bb{C}^*$, the global intersection cohomology agrees with the stalk intersection cohomology at $0$ (see for instance \cite{htt}*{(13.2.40)}). The conclusion $\chi_{0,j}=(-1)^{d_j}{n\choose j}$ then follows from Corollary~\ref{cor:IC-general-mat} by taking Euler characteristic by setting $q=-1$ (cf. Corollary \ref{cor:repfin}). If $i>0$ then we can compute the invariants after restricting to the open set $X\setminus V_{i-1}$. This space is locally isomorphic to $X^{m-i,n-i}\times B$ for a smooth base $B$, via an isomorphism compatible with the stratification: $X_j=X^{m,n}_j$ corresponds to $X^{m-i,n-i}_{j-i}\times B$ (see for instance \cite{lor-rai}*{Section~2H}). Using the fact that $IC^{\bullet}_{V^{m-i,n-i}_{j-i}\times B} = IC^{\bullet}_{V^{m-i,n-i}_{j-i}} \boxtimes IC^{\bullet}_B$, and $IC^{\bullet}_B = \bb{C}[\dim B]$, it follows that if $(0,b)\in X^{m-i,n-i}_{j-i}\times B$ is the point corresponding to $x_i\in X^{m,n}_i$ then
 \[ h^k\left(IC^{\bullet}_{V^{m,n}_{j},x_i}\right) = h^k\left(IC^{\bullet}_{V^{m-i,n-i}_{j-i}\times B,(0,b)}\right) = h^{k-dim(B)}\left(IC^{\bullet}_{V^{m-i,n-i}_{j-i},0}\right).\]
By taking Euler characteristic, we then obtain 
\[ \chi_{i,j}^{m,n} = \chi_{0,j-i}^{m-i,n-i}\cdot(-1)^{\dim B}\]
and the desired formula for $\chi_{i,j}^{m,n}$ follows by induction on the size of the matrix space. Alternatively, using properties 1. and 3. for local Euler obstructions from \cite{macpherson}*{Section~3}, we get that $e^{m,n}_{i,j} = e^{m-i,n-i}_{0,j-i}$.
\end{proof}

\begin{remark}\label{rem:other} We conclude this section with several remarks:
\begin{itemize}
\item[(a)] Via the identification of $X^{m,n}$ with the big opposite cell in the Grassmannian $\bb{G}(n,m+n)$, it follows from \cite{las-sch}*{Section~11} that the intersection cohomology groups in Corollary~\ref{cor:IC-general-mat} (as well as their local versions) are computed by Kazhdan--Lusztig polynomials. The resulting Gaussian polynomials in our formulas are then precisely the ones appearing in \cite{las-sch}*{Lemme~10.1}.

\item[(b)] As mentioned in Section~\ref{subsec:Gaussian}, the Gaussian polynomials are Poincar\'e polynomials of Grassmann varieties. The determinantal variety $V_p$ has a small resolution given by a vector bundle over the Grassmannian $\bb{G}(p,n)$ (see also \cite{zel-small}, \cite{weyman}*{Proposition~6.1.1}, \cite{per-rai}*{Section~5.2}), which gives another explanation for the formula in Corollary~\ref{cor:IC-general-mat}.

\item[(c)] If we write $\mc{F}$ for the Fourier transform (see \cite{htt}*{Section~3.2.2}, \cite{raicu-dmods}*{Section~2.5}, \cite{lor-wal}*{Section~4.3}), then we have $\mc{F}(D_p) \cong D_{n-p}$ for all $p$. The formula in (\ref{eq:deRham-Dp}) can then be recovered from the case $t=0$ of \cite{lor-rai}*{Theorem 1.1} via an iteration of \cite{htt}*{Proposition 3.2.6}, which implies that
\[
H^{k}(\pi_+ (M)) \cong H^{k}(Li^*\mc{F}(M)),
\]
where $\pi : X \to \{0\}$ is the projection and $i : \{0\} \to X$ the inclusion.

\item[(d)] The parity-vanishing of de Rham cohomology in (\ref{eq:deRham-Dp}) follows from the vanishing of odd-dimensional intersection cohomology groups for spherical varieties  \cite{bri-josh}.
\end{itemize}
\end{remark}

\section{Symmetric matrices}
\label{sec:symmetric}

In this section we consider the space $X=X^{n,\symm}$ of $n\times n$ symmetric matrices. We let $F$ be a vector space with $\dim(F)=n$, let $S = \Sym(\Sym^2 F)$, and make the identification $X=\Spec(S)$. Let $G=\GL(F)$ and consider its natural action on $X$, so that the rank stratification agrees with the orbit stratification for the $G$-action. The dimensions of the strata are given by
\begin{equation}\label{eq:dp-symm}
d_p = \frac{p\cdot(2n-p+1)}{2}\text{ for }p=0,\cdots,n.
\end{equation}
By (\ref{eq:extsymm}) the modules of differential forms are described by
\begin{equation}\label{eq:omegasym}
\Omega^i_X = \bw^i\Omega^1_X = \bigoplus_{\mu\in \mc{Y}(2i)} \SS_\mu F \oo S.
\end{equation}

We consider as in Section~\ref{sec:general-matrices} the simple modules $D_p:=\mc{L}(V_p,X)$. Their decomposition as a direct sum of irreducible $G$-representations is computed in \cite{raicu-dmods}*{Theorem~4.1}, and is given as follows.

\begin{enumerate}
    \item If $n-p$ is odd then
\begin{equation}\label{eq:symm-Dp-odd}
 D_p = \bigoplus_{\ll \in \mc{C}^1(p)}\bb{S}_{\ll} F^{\vee}
\end{equation}
where
\[\mc{C}^1(p) =\{\ll\in\bb{Z}^{n}_{\dom}:\ll_i\overset{(\opmod\ 2)}{\equiv} 0\rm{ for }i=1,\cdots,n,\ \ll_{n-p}\geq n-p+1\geq\ll_{n-p+2}\}.\]
    \item If $n-p$ is even then
\begin{equation}\label{eq:symm-Dp-even}
 D_p = \bigoplus_{\ll \in \mc{C}^2(p)}\bb{S}_{\ll} F^{\vee}
\end{equation}
where
\[\mc{C}^2(p) =\left\{\ll\in\bb{Z}^{n}_{\dom}:\ll_i\overset{(\opmod\ 2)}{\equiv}
\begin{cases}
 1 & \rm{for }i\leq n-p\\
 0 & \rm{for }i\geq n-p+1
\end{cases},
\ll_{n-p}\geq n-p+1,\ll_{n-p+1}\leq n-p
\right\}.\]
\end{enumerate}

\begin{theorem}\label{thm:deRham-Dp-symm} 
We let $m=\lfloor n/2 \rfloor$, and for $0\leq p \leq n$, we set $s=\lfloor p/2\rfloor$, and
\[ \epsilon = \begin{cases}
1 & \text{if }p=2s\text{ is even and }n=2m+1\text{ is odd}; \\
0 & \text{otherwise}.
\end{cases}
\]
We have
\begin{equation}\label{eq:invariant-deRham-Dp-symm}
 \sum_{i\geq 0} \dim(\Omega^i_X\oo D_p)^{G} \cdot q^i = {m+\epsilon \choose s}_{q^4} \cdot q^{{n-p+1\choose 2}}.
\end{equation}
In particular, the differentials in the invariant de Rham complex $DR(D_p)^{G}$ are identically zero, and
\begin{equation}\label{eq:deRham-Dp-symm}
\sum_{i\geq 0} h^i_{dR}(D_p) \cdot q^i =  {m+\epsilon \choose s}_{q^4} \cdot q^{{n-p+1\choose 2}}.
\end{equation}
\end{theorem}

\begin{proof} We separate our analysis into two cases, according to the parity of $n-p$, and proceed as in the proof of Theorem~\ref{thm:deRham-Dp}. We have using \eqref{eq:Schur-lemma} and \eqref{eq:extsymm} that $\dim(\Omega^i_X\oo D_p)^{G}$ is equal to the number of partitions $\ll\in\mc{Y}(2i)$ with the property that $\bb{S}_{\ll}F^{\vee}$ appears as a summand of $D_p$. We consider any such $\ll$ and write $r$ for its Durfee size.

If $n-p$ is even then we have by \eqref{eq:symm-Dp-even} that $\ll\in\mc{C}^2(p)$, hence $\ll_{n-p}\geq n-p+1$, $\ll_{n-p+1}\leq n-p$, forcing $r=n-p$. With the notation \eqref{eq:def-Y}, we have $\alpha\leq (p^{n-p})$, and the condition $\ll\in\mc{C}^2(p)$ is then equivalent to the fact that both $\alpha$ and $\alpha'$ have even parts. The choice of $\alpha$ is then equivalent to that of a partition $\beta\leq\left( s^{(n-p)/2}\right)$ with $4|\beta|=|\alpha|$ (recall that $s=\lfloor p/2\rfloor$), via the rule
\[\alpha_{2i-1}=\alpha_{2i} = 2\beta_i\text{ for all }i\geq 1.\]
Using \eqref{eq:qbin-genfun}, we conclude that
\[\sum_{i\geq 0} \dim(\Omega^i_X\oo D_p)^{G} \cdot q^i = {\frac{n-p}{2} + s \choose s}_{q^4} \cdot q^{\frac{(n-p)\cdot(n-p+1)}{2}}.\]
To see that this agrees with \eqref{eq:invariant-deRham-Dp-symm}, we first note that $n-p$ even implies that $\epsilon=0$. If $n=2m$ is even then $p=2s$, hence $\frac{n-p}{2} + s = m$, while if $n=2m+1$ is odd then $p=2s+1$, and we have again $\frac{n-p}{2} + s = m$.

If $n-p$ is odd then we have by \eqref{eq:symm-Dp-odd} that $\ll\in\mc{C}^1(p)$. The condition $\ll_{n-p}\geq n-p+1\geq\ll_{n-p+2}$ implies that $r\in\{n-p,n-p+1\}$, so we have two type of contributions to $(\Omega^i_X\oo D_p)^{G}$. Suppose first that $r=n-p$, and let $\alpha$ as in \eqref{eq:def-Y}. We have $\alpha\leq(p^{n-p})$, and the condition $\ll\in\mc{C}^1(p)$ is again equivalent to the fact that both $\alpha$ and $\alpha'$ have even parts. The choice of $\alpha$ is equivalent as before to that of a partition $\beta\leq\left( s^{(n-p-1)/2}\right)$ with $4|\beta|=|\alpha|$, and varying $\alpha$ and using \eqref{eq:qbin-genfun} we get a contribution to $\sum_{i\geq 0} \dim(\Omega^i_X\oo D_p)^{G}\cdot q^i$ of
\begin{equation}\label{eq:1st-contr-symm}
{\frac{n-p-1}{2} + s \choose s}_{q^4} \cdot q^{\frac{r^2+r}{2}} = {\frac{n-p-1}{2} + s \choose s}_{q^4} \cdot q^{\frac{(n-p)\cdot(n-p+1)}{2}}.
\end{equation}
Suppose next that $r=n-p+1$, so that $\alpha\leq\left((p-1)^{n-p+1}\right)$, $\alpha$ has odd parts, and $\alpha'$ has even parts. If we define $\ol{\alpha}$ by $\ol{\alpha}_i = \alpha_i-1$ for $i=1,\cdots,r$ then $|\ol{\alpha}|=|\alpha|-r$, $\ol{\alpha}\leq\left((p-2)^{n-p+1}\right)$, and both $\ol{\alpha}$ and $\ol{\alpha}'$ have even parts. The choice of $\ol{\alpha}$ is then equivalent to that of $\beta\leq\left((s-1)^{(n-p+1)/2}\right)$ with $4|\beta|=|\ol{\alpha}|$. We get from \eqref{eq:qbin-genfun} the second contribution to $\sum_{i\geq 0} \dim(\Omega^i_X\oo D_p)^{G}\cdot q^i$ of
\begin{equation}\label{eq:2nd-contr-symm}
{\frac{n-p+1}{2} + s-1 \choose s-1}_{q^4} \cdot q^{\frac{r^2+r}{2}+r} = {\frac{n-p+1}{2} + s-1 \choose s-1}_{q^4} \cdot (q^4)^{\frac{n-p+1}{2}} \cdot q^{\frac{(n-p)\cdot(n-p+1)}{2}}.
\end{equation}
Using \eqref{eq:qbin-is-bin}, and the recursion \eqref{eq:qbin-recursion} with $a=(n-p+1)/2+s$ and $b=(n-p+1)/2$, we get
\[{\frac{n-p-1}{2} + s \choose s}_{q^4} + {\frac{n-p+1}{2} + s-1 \choose s-1}_{q^4} \cdot (q^4)^{\frac{n-p+1}{2}} = {\frac{n-p+1}{2} + s \choose s}_{q^4},\]
which combined with \eqref{eq:1st-contr-symm} and \eqref{eq:2nd-contr-symm} implies that
\[\sum_{i\geq 0} \dim(\Omega^i_X\oo D_p)^{G} \cdot q^i = {\frac{n-p+1}{2} + s \choose s}_{q^4} \cdot q^{\frac{(n-p)\cdot(n-p+1)}{2}}.\]
To see that this agrees with \eqref{eq:invariant-deRham-Dp-symm}, we consider two cases. If $n=2m+1$ then $p=2s$ and hence $\epsilon=1$, showing that $\frac{n-p+1}{2} + s = m+1=m+\epsilon$. If $n=2m$ then $p=2s+1$ and hence $\epsilon=0$, which shows that $\frac{n-p+1}{2} + s = m = m+\epsilon$, as desired.

It follows from \eqref{eq:invariant-deRham-Dp-symm} that there are no non-zero consecutive terms in the invariant de Rham complex, which forces the differentials to be zero,  which in turn implies \eqref{eq:deRham-Dp-symm} and concludes our proof. 
\end{proof}

We can now derive the formula for the intersection cohomology of symmetric determinantal varieties.

\begin{corollary}\label{cor:IC-symmetric-mat}
 With the notation in Theorem~\ref{thm:deRham-Dp-symm}, we have
 \[ \sum_{i\in\bb{Z}} h^i\left(IC^{\bullet}_{V_p}\right) \cdot q^i = q^{-d_p} \cdot {m+\epsilon \choose s}_{q^4}.\]
\end{corollary}

\begin{proof}
 Since $\dim(X)={n+1\choose 2}$, and $h^{i-\dim(X)}\left(IC^{\bullet}_{V_p}\right) = h^i_{dR}(D_p)$, it suffices by Theorem~\ref{thm:deRham-Dp-symm} to show that
 \[{n-p+1 \choose 2} - {n+1\choose 2} = -d_p,\]
 which follows from \eqref{eq:dp-symm}.
\end{proof}

We end this section with the calculation of the invariants $\chi_{i,j}$ and $e_{i,j}$.

\begin{corollary}\label{cor:chi-eij-symmetric}
 For $0\leq i\leq j\leq n$ we have that
\[ \chi_{i,j} = (-1)^{d_j}\cdot\displaystyle{\lfloor \frac{n-i}{2}\rfloor + \epsilon \choose \lfloor\frac{j-i}{2}\rfloor},\text{ where }
     \epsilon = \begin{cases}
    1 & \text{if }(j-i)\text{ is even and }(n-i)\text{ is odd}; \\
    0 & \text{otherwise}.
    \end{cases}
\]
Moreover, the local Euler obstructions are given by
 \[e_{i,j} = \begin{cases}
 0 & \text{ if }n-i\text{ is even and }n-j\text{ is odd}; \\
 \displaystyle{\lfloor\frac{n-i}{2}\rfloor \choose \lfloor\frac{j-i}{2}\rfloor} & \text{otherwise}.
 \end{cases}.\]
\end{corollary}

\begin{proof}
 As in the proof of Corollary~\ref{cor:chi-eij-general}, the formula for $\chi_{0,j}$ follows from Corollary~\eqref{cor:IC-symmetric-mat} by taking $p=j$ and $q=-1$, and observing that via the $\bb{C}^*$-action, the global intersection cohomology agrees with the local one at the origin. For $i>0$ we do induction on $n$, using the fact that $\chi_{i,j}^{n,\symm}=\chi_{0,j-i}^{n-i,\symm}$ and $e_{i,j}^{n,\symm}=e_{0,j-i}^{n-i,\symm}$, which follows as in the proof of Corollary~\ref{cor:chi-eij-general} using that $X^{n,\symm}\setminus V_{i-1}$ is locally isomorphic to $X^{n-i,\symm}\times B$ for a smooth base $B$.
 
 It remains to check the formulas for $e_{0,j}$, for which we apply \eqref{eq:local-index} with $i=0$. Using the Theorem on characteristic cycles from the Introduction, we note that when $n-j$ is even, $m_{i,j}=0$ for all $i\neq j$, hence
 \[ e_{0,j} = (-1)^{d_j}\cdot \chi_{0,j} = \displaystyle{\lfloor\frac{n}{2}\rfloor \choose \lfloor\frac{j}{2}\rfloor},\]
 where the last equality uses the fact that $\epsilon=0$ (if $j$ was even and $n$ odd then $n-j$ would be odd). 
 
 Suppose now that $n-j$ is odd, so that $m_{j-1,j}=m_{j,j}=1$ and $m_{i,j}=0$ for $i<j-1$. We have by \eqref{eq:local-index} that
 \[ \chi_{0,j} = (-1)^{d_{j-1}}\cdot e_{0,j-1} + (-1)^{d_j}\cdot e_{0,j}.\]
 Since we already know $\chi_{0,j}$ and $e_{0,j-1}$, and since \eqref{eq:dp-symm} implies $d_j-d_{j-1} = n+1-j$ is even, we get
 \[e_{0,j} = {\lfloor \frac{n}{2}\rfloor + \epsilon \choose \lfloor\frac{j}{2}\rfloor} - 
 {\lfloor \frac{n}{2}\rfloor \choose \lfloor\frac{j-1}{2}\rfloor}.
 \]
 If $n$ is even then $j$ is odd, hence $\epsilon=0$ and $\lfloor\frac{j}{2}\rfloor=\lfloor\frac{j-1}{2}\rfloor$, and therefore $e_{0,j}=0$, as desired. If $n$ is odd then $j$ is even, hence $\epsilon=1$ and $\lfloor\frac{j-1}{2}\rfloor=\lfloor\frac{j}{2}\rfloor-1$, so
 \[e_{0,j} = {\lfloor \frac{n}{2}\rfloor + 1 \choose \lfloor\frac{j}{2}\rfloor} - 
 {\lfloor \frac{n}{2}\rfloor \choose \lfloor\frac{j}{2}\rfloor-1} = {\lfloor \frac{n}{2}\rfloor \choose \lfloor\frac{j}{2}\rfloor},
 \]
 which concludes our proof.
\end{proof}

\section{Skew-symmetric matrices} \label{sec:skew}

In this section we consider the space $X=X^{n,\skew}$ of $n\times n$ skew-symmetric matrices. We let $F$ be a vector space with $\dim(F)=n$, let $m=\lfloor{n/2}\rfloor$, $S = \Sym(\bigwedge^2 F)$, and identify $X=\Spec(S)$. We write $G=\GL(F)$ and consider its natural action on $X$, so that the rank stratification on $X$ agrees with the orbit stratification for the $G$-action. The dimensions of the strata are given by
\begin{equation}\label{eq:dp-skew}
d_p = p\cdot(2n-2p-1)\text{ for }p=0,\cdots,m.
\end{equation}
By (\ref{eq:extskew}) the modules of differential forms are described by
\begin{equation}\label{eq:omegaskew}
\Omega^i_X = \bw^i\Omega^1_X = \bigoplus_{\mu\in \mc{Z}(2i)} \SS_\mu F \oo S.
\end{equation}
The decomposition into irreducible $G$-representations of the modules $D_p=\mc{L}(V_p, X)$ is (see \cite{raicu-dmods}*{Section~6})
\begin{equation}\label{eq:decomp-Dp-skew}
 D_p = \bigoplus_{\ll \in \mc{B}(p)}\bb{S}_{\ll} F^{\vee},
\end{equation}
where if $n=2m$ is even then
\[
\mc{B}(p) = \{\ll\in\bb{Z}^{2m}_{\dom}:\ll_{n-2p}\geq n-2p-1,\ll_{n-2p+1}\leq n-2p,\ll_{2i-1}=\ll_{2i}\rm{ for all }i\},
\]
and if $n=2m+1$ is odd then
\[
\mc{B}(p) = \{\ll\in\bb{Z}^{2m+1}_{\dom}:\ll_{n-2p}=n-2p-1,\ll_{2i-1}=\ll_{2i}\rm{ for }i\leq m-p,\ll_{2i}=\ll_{2i+1}\rm{ for }i>m-p\}.
\]
We note that for $p=m$ we have $D_m=S$ and the formula in (\ref{eq:decomp-Dp-skew}) is classical \cite{weyman}*{Proposition 2.3.8}.

\begin{theorem}\label{thm:deRham-Dp-skew} For $0\leq p \leq m$, we have
\begin{equation}\label{eq:invariant-deRham-Dp-skew}
\sum_{i\geq 0} \dim(\Omega^i_X\oo D_p)^{G} \cdot q^i = {m \choose p}_{q^4} \cdot q^{{n\choose 2}-p(2n-2p-1)}.
\end{equation}
In particular, the differentials in the invariant de Rham complex $DR(D_p)^{G}$ are identically zero, and
\begin{equation}\label{eq:deRham-Dp-skew}
\sum_{i\geq 0} h^i_{dR}(D_p) \cdot q^i =  {m \choose p}_{q^4} \cdot q^{{n\choose 2}-p(2n-2p-1)}.
\end{equation}
\end{theorem}

\begin{proof} If we combine \eqref{eq:Schur-lemma}, \eqref{eq:extskew}, and \eqref{eq:decomp-Dp-skew} then it follows that $\dim(\Omega^i_X\oo D_p)^{G}$ is equal to the number of partitions $\ll\in\mc{Y}(2i) \cap \mc{B}(p)$. We consider any such $\ll$ and write $r$ for its Durfee size, and divide our analysis according to the parity of $n$.

If $n=2m+1$ then the condition $\lambda_{n-2p}=n-2p-1$ forces $r=n-2p-1=2(m-p)$. If we let $\alpha$ as in \eqref{eq:def-Z}, then we have $\alpha\leq\left((2p)^{2(m-p)}\right)$ and the condition $\lambda\in \mc{B}(p)$ is equivalent to the fact that both $\alpha$ and $\alpha'$ have even parts. Therefore, $\alpha$ corresponds to a partition $\beta\leq (p^{m-p})$ with $4|\beta|=|\alpha|$, as in the proof of Theorem~\ref{thm:deRham-Dp-symm}, and we get
\[\sum_{i\geq 0} \dim(\Omega^i_X\oo D_p)^{G} \cdot q^i = {m \choose p}_{q^4} \cdot q^{\frac{r^2+r}{2}}.\]
This agrees with \eqref{eq:invariant-deRham-Dp-symm} because
\[{n\choose 2}-p(2n-2p-1) = m(2m+1)-p(4m-2p+1) = (m-p)(2m-2p+1) = \frac{r(r+1)}{2}.\]

Suppose now that $n=2m$, where the inequalities $\ll_{n-2p}\geq n-2p-1,\ll_{n-2p+1}\leq n-2p$ imply that $r\in\{n-2p-1,n-2p\}$. Consider first the case $r=n-2p-1$ and let $\a$ as in \eqref{eq:def-Z}. We have $\alpha\leq\left((2p)^{n-2p-1}\right)$, and since $r$ is odd we get
\[ r+\alpha_r = \ll_{r} = \ll_{r+1} = r,\]
forcing $\alpha_r = 0$. It follows that $\alpha\leq\left((2p)^{2m-2p-2}\right)$, and both $\alpha,\alpha'$ have even parts, and using the standard correspondence with partitions $\beta\leq \left(p^{m-p-1}\right)$ we get a contribution to $\sum_{i\geq 0} \dim(\Omega^i_X\oo D_p)^{G} \cdot q^i$ of
\begin{equation}\label{eq:1st-contr-skew}
{m-1 \choose p}_{q^4} \cdot q^{\frac{r^2+r}{2}} = {m-1 \choose p}_{q^4} \cdot q^{(m-p)\cdot(n-2p-1)}.
\end{equation}
Consider next the case $r=n-2p$ and let $\a$ as in \eqref{eq:def-Z}, so that $\alpha\leq \left((2p)^{n-2p}\right)$. Since $r$ is even, we have
\[ r = \ll_{r+1} = \ll_{r+2} = \alpha'_1,\]
so $\alpha_i\geq 1$ for $1\leq i\leq r$. We define $\ol{\alpha}$ by $\ol{\alpha}_i=\alpha_i-1$, so $|\ol{\alpha}|=|\alpha|-r$ and $\ol{\alpha}\leq \left((2p-1)^{n-2p}\right)$. Now $\ol{\alpha}$ and $\ol{\alpha}'$ have even parts, so $\ol{\alpha}$ corresponds to a partition $\beta \leq \left((p-1)^{m-p}\right)$ with $4|\beta|=|\ol{\alpha}|$. We get a contribution to $\sum_{i\geq 0} \dim(\Omega^i_X\oo D_p)^{G} \cdot q^i$ of
\begin{equation}\label{eq:2nd-contr-skew}
{m-1 \choose p-1}_{q^4} \cdot q^{\frac{r^2+r}{2}+r} = {m-1 \choose p-1}_{q^4} \cdot (q^4)^{(m-p)} \cdot q^{(m-p)\cdot(n-2p-1)}.
\end{equation}
Summing \eqref{eq:1st-contr-skew} and \eqref{eq:2nd-contr-skew} and using \eqref{eq:qbin-recursion} with $a = m$ and $b=m-p$ we get
\[\sum_{i\geq 0} \dim(\Omega^i_X\oo D_p)^{G} \cdot q^i = {m \choose p}_{q^4} \cdot q^{(m-p)\cdot (n-2p-1)}.\]
This agrees with \eqref{eq:invariant-deRham-Dp-symm} because
\[{n\choose 2}-p(2n-2p-1) = m(2m-1)-p(4m-2p-1) = (m-p)(2m-2p-1).\]

Finally, \eqref{eq:deRham-Dp-skew} follows from \eqref{eq:invariant-deRham-Dp-symm} because $DR(D_p)^G$ contains no two consecutive non-zero terms.
\end{proof}

As for the other matrix spaces, we get a formula for intersection cohomology as follows.

\begin{corollary}\label{cor:IC-skew-mat}
 With the notation in Theorem~\ref{thm:deRham-Dp-skew}, we have
 \[ \sum_{i\in\bb{Z}} h^i\left(IC^{\bullet}_{V_p}\right) \cdot q^i = q^{-d_p} \cdot {m \choose p}_{q^4}.\]
\end{corollary}

\begin{proof}
 Since $\dim(X)={n\choose 2}$, and $h^{i-\dim(X)}\left(IC^{\bullet}_{V_p}\right) = h^i_{dR}(D_p)$, the conclusion follows from \eqref{eq:dp-skew} and \eqref{eq:deRham-Dp-skew}.
\end{proof}

Finally, we can derive the formulas for $\chi_{i,j}$ and $e_{i,j}$ from the Introduction.

\begin{corollary}\label{cor:chi-Euler-skew}
 Let $m=\lfloor n/2\rfloor$. For $0\leq i\leq j\leq m$ we have that
 \[\chi_{i,j} = (-1)^{d_j}\cdot{m-i\choose j-i},\quad e_{i,j} = {m-i \choose j-i}.\]
\end{corollary}

\begin{proof}
 Since the characteristic cycles of the modules $D_p$ are irreducible, we get from \eqref{eq:local-index} that $\chi_{i,j}=(-1)^{d_j}e_{i,j}$, so it suffices to prove the formula for $\chi_{i,j}$. When $i=0$ we set $q=-1$ and $p=j$ in Corollary~\eqref{cor:IC-skew-mat}, and use the identification between global intersection cohomology and the stalk at $0$. When $i>0$ we use the inductive structure as in the case of general and symmetric matrices, which gives $\chi^{n,\skew}_{i,j}=\chi^{n-2i,\skew}_{0,j-i}$.
\end{proof}

\section*{Acknowledgements}
 Experiments with the computer algebra software Macaulay2 \cite{M2} have provided numerous valuable insights. Raicu acknowledges the support of the National Science Foundation Grant No.~1901886. We are grateful to the anonymous referee for many useful suggestions and references.

\end{document}